\documentclass{amsart}
\usepackage{amssymb}

\usepackage[top=1 in,bottom=1 in, left=1 in, right= 1 in]{geometry}

\def\R{{\mathbb R}}
\def\C{{\mathbb C}}

\def\<{\langle}
\def\>{\rangle}
\def\P{{\mathbb P}}

\def\E{{\mathbb E}}

\def\<{\langle}
\def\>{\rangle}

\newcommand{\bel}{\begin{equation}\label}
\newcommand{\ee}{\end{equation}}
      \newtheorem{theorem}{Theorem}[section]
       \newtheorem{proposition}[theorem]{Proposition}

\theoremstyle{definition}

\date{printed \today\ file \jobname.tex}

\begin{document}

\title{Linearity of regression for weak records, revisited}
\author{Rafa{\l} Karczewski, Jacek Weso\l owski}
\address{Rafa{\l} Karczewski\\
Wydzia{\l} Matematyki i Nauk Informacyjnych\\
Politechnika Warszawska\\
Koszykowa 75\\
00-662 Warszawa, Poland}
\email{karczewskir@student.mini.pw.edu.pl}
\address{Jacek Weso{\l}owski\\
Wydzia{\l} Matematyki i Nauk Informacyjnych\\
Politechnika Warszawska\\
Koszykowa 75\\
00-662  Warszawa, Poland}
\email{wesolo@mini.pw.edu.pl}
\begin{abstract}
Since many years characterization of distribution by linearity of regression of non-adjacent weak records $\E(W_{i+s}|W_i)=\beta_1W_i+\beta_0$ for discrete observations has been known to be a difficult question. L\'opez-Bl\'azquez (2004) proposed an interesting idea of reducing it to the adjacent case and claimed to have the characterization problem completely solved. We will explain that, unfortunately, there is a flaw in the proof given in that paper. This flaw is related to fact that in some situations the operator responsible for reduction of the non-adjacent case to the adjacent one is not injective. The operator is trivially injective when $\beta_1\in(0,1)$. We show that when $\beta_1\ge 1$ the operator is injective when $s=2,3,4$. Therefore in these cases the method proposed by L\'opez-Bl\'azquez is valid. We also show that the operator is not injective when $\beta_1\ge 1$ and $s\ge 5$. Consequently, in this case the reduction methodology does not work and thus the characterization problem remains open.
\end{abstract}
\maketitle
\section{Introduction}
The issue of characterization of the common distribution of a sequence $(X_n)_{n\ge 1}$ of iid variables by linearity of regression of records $\E(R_m|R_n)=\beta_1R_n+\beta_0$ for $m\ne n$ has attracted the attention of researchers since the seventies in the last century, when Nagaraja (1977), assuming that the common distribution of $X_n$'s is continuous and following methods developed by Ferguson (1967) for order statistics, characterized the triplet of exponential, power and Pareto type distributions in the case $m=n+1$. In Nagaraja (1988) the case $m=n-1$ for continuous distribution was solved by reducing the problem to the one for order statistics. As a result another triplet of distributions was characterized. The characterization in the case $m=n+2$ was done in Ahsanullah and Weso\l owski (1998) through reducing the problem to second order ordinary differential equation and a careful look at its probabilistic solutions. The characterization issue for continuous distributions was finally resolved in the general case of linearity of regression for non-adjacent records in Dembi\'nska and Weso\l owski (2001) by using integrated Cauchy functional equation in case $m>n$ and in case $m<n$ by reducing the problem to an analogous problem for order statistics, the latter being solved by a similar method earlier in Dembi\'nska and Weso\l owski (1998). Since that time the case of continuous parent distribution has been studied further e.g. for generalized order statistics and for other patterns of regression functions. For these and related issues see e.g. L\'opez-Bl\'azquez and Moreno-Rebollo (1997), Bieniek and Szynal (2003), Cramer, Kamps and Keseling (2004), Gupta and Ahsanullah (2004), Bairamov, Ahsanullah and Pakes (2005), Bieniek (2007), Bieniek (2009), Yanev (2012) and Ahsanullah and Hamedani (2013), Beg, Ahsanullah and Gupta (2013).

In the case of discrete distribution instead of records $(R_n)$, which are defined through a strict inequality, it is more natural to consider weak records $(W_n)$, which are defined by "$\ge$" relation. That is, a repetition of the last weak record is the next weak record, while for regular records repetitions of records are discarded. In this case, the issue of characterization of the distribution of $X_n$'s through  linearity of regression $\E(W_m|W_n)=\beta_1W_n+\beta_0$ for $m\ne n$ seems not to be related to the methods developed in the continuous case. In particular, under natural assumption that the support of the common law of $X_n$'s is a set of the form $\{0,1,\ldots,N\}$ with $N\le \infty$ we see that in the case of $m<n$, due to monotonicity of $(W_n)$ sequence, we have $\beta_0=0$. To the best of our knowledge, under this assumption ($m<n$) the characterization was obtained only in two special cases: $\E(W_1|W_2)=\beta_1W_2$ for $\beta_1>0$ in Lop\'ez-Bl\'azquez and Weso\l owski (2001) and $\E(W_m|W_n)=\tfrac{m}{n}W_n$ in Lop\'ez-Bl\'azquez and Weso\l owski (2004).

For the case $m>n$ the characterization of distribution of $X_n$'s was first given in Stepanov (1994) for $m=n+1$  with an  improvement in Weso\l owski and Ahsanullah (2001) - see also related papers Aliev (1998), Danielak and Dembi\'nska (2007), Dembi\'nska and L\'opez-Bl\'azquez (2005). The case $m=n+2$ for $\beta_1=1$ was considered in Aliev (2001) and for general $\beta_1$ in Weso\l owski and Ahsanullah  (2001), where an approach via solution of a non-linear difference equation was applied. In this way a triplet of geometric and negative hypergeometric distributions of the first and second kind was characterized. For the general case $m>n$ L\'opez-Bl\'azquez (2004) (we refer to this paper by LB in the sequel) proposed an intriguing idea of reduction of the problem to the adjacent case of $m=n+1$, for which the solution has been already known. However, as it will be explained below, this interesting approach is not as universal as it is claimed in that paper. It appears that there are some inaccuracies in the proof in the case $\beta_1\ge 1$, that is when $N=\infty$. When we encountered these inaccuracies we were rather confident that it would be possible to overcome them while preserving this brilliant idea of reduction to the adjacent case $m=n+1$. As we will see, this can be done only if $0<m-n\le 4$. Unfortunately, for higher distances between $m$ and $n$ the idea introduced in LB does not work. Therefore the characterization in the case $m>n+4$ and $\beta_1\ge 1$ still remains an open problem.

Finally, let us mention that the issue of characterization of discrete distributions by linearity of regression of ordinary records $\E(R_m|R_n)=\beta_1R_n+\beta_0$ has also been considered in the literature. If $m>n$ only characterizations of tails of distribution were eventually obtained, see e.g. Srivastava (1979), Kirmani and Alam (1980), Ahsanullah and Holland (1984), Korwar (1984),  Rao and Shanbhag (1986), Nagaraja, Sen and Srivastava (1989), Huang and Su (1999). If $m<n$ no elegant characterization seems to be possible, see L\'opez-Bl\'azquez and Weso\l owski (2004), except the case $m=1$, $n=2$, see  L\'opez-Bl\'azquez and Weso\l owski (2001) and Franco and Ruiz (2001).

\section{Passing from the non-adjacent case to the adjacent one is problematic}
We consider a sequence $(X_n)$ of iid random variables having the common distribution ${\bf p}=(p_k)$ supported on $\{0,1,\ldots,N\}$, $N\le \infty$. That is, $p_k=\P(X_1=k)$, and we also write $q_k=\P(X_1\ge k)$, $k=0,1,\ldots,N$. For such a sequence, we consider the respective sequence of weak records $(W_n)$ which is defined as follows: Let $T_1=1$ and $T_n=\inf\{k>T_{n-1}:\,X_k\ge X_{T_{n-1}}\}$, $n>1$. Then $W_n=X_{T_n}$, $n\ge 1$. The joint distribution of the first $n$ weak records can be easily derived as
$$
\P(W_1=k_1,\ldots,W_n=k_n)=p_{k_n}\prod_{j=1}^{n-1}\,\tfrac{p_{k_j}}{q_{k_j}},\quad 0\le k_1\le \ldots\le k_n.
$$
Weak records were introduced in Vervaat (1973) and since then are one of important models for ordered discrete random variables. Their basic properties can be found in any monograph on records, e.g., Ch. 2.8 of Arnold, Balakrishnan and Nagaraja (1998), Ch. 16 of Nevzorov (2001) or in Ch. 6.3. of relatively recent monograph Ahsanullah and Nevzorov (2015). It is well-known that weak records form a homogeneous Markow chain with the transition probability of the form
$$
\P(W_n=k_n|W_{n-1}=k_{n-1})=\tfrac{p_{k_n}}{q_{k_{n-1}}},\quad k_n\ge k_{n-1}\ge 0.
$$
Therefore, for $m<n$
$$
\P(W_n=k_n|W_m=k_m)=\sum_{k_m\le k_{m+1}\le\ldots\le k_{n-1}\le k_n} \prod_{i=m}^{n-1}\,\tfrac{p_{k_{i+1}}}{q_{k_i}},\quad k_n\ge k_m\ge 0.
$$

For fixed positive integers $i,s$ we will be interested in conditional expectation $\E(W_{i+s}|W_i)$. Therefore we need to assume that ${\bf p}$ is such that this conditional expectation is finite. Since the conditional distribution of $W_{i+s}|W_i$ does not depend on $i$, we will denote the set of distributions ${\bf p}$ for which $\E(W_{i+s}|W_i)$ is finite by $\mathcal{M}_s$.

Let us consider a family $C_s$ of discrete distributions ${\bf p}=(p_k)_{k\ge 0}\in \mathcal{M}_s$, concentrated on $\lbrace 0,1,2, \dots , N \rbrace$ ($N \le \infty$) with  property that if the common law of iid random variables $(X_n)_{n\ge 1}$ belongs to $C_s$ then the regression of weak records $\mathbb{E} \left( W_{i+s} | W_i \right)$ is linear. It is known that $C_1 \subseteq C_s$ for all $s\ge 1$. We are interested in the opposite inclusion. In LB it is claimed that the opposite implication holds true, however the proof of this inclusion given in there is not correct. We will explain why it is not correct, then improve the method proposed in LB to show that the inclusion holds true for $s=2,3,4$ and finally we will show that the method fails for $s \ge 5$.

Before we state the result from LB we need to introduce some notation. Let ${\bf v}=\left( v(0), v(1), \hdots, v(N) \right) \in \mathbb{C}^{N+1}$ (for $N=\infty$, ${\bf v}=\left( v(0),v(1), \hdots \right)$). Let us define a linear operator:
\[
A : D(A) \longrightarrow \mathbb{C}^{N+1} \hspace{2pt} ; \hspace{5pt} Av(l)=\frac{1}{q_l} \sum_{k=l}^N v(k) p_k,\quad l=0,1,\ldots,N,
\]
where
\[
D(A)=\lbrace {\bf v} \in \mathbb{C}^{N+1} : \sum_{k=0}^{N} \left| v(k) \right| p_k < \infty \rbrace.
\]
We also define the domain of composition of operator $A$ with itself since we will need that later on:
\[
D(A^m)= \lbrace {\bf v} \in D(A) : A^k {\bf v} \in D(A) \mbox{ for } k=1,\hdots, m-1 \rbrace \mbox{ for } m \ge 2,
\]
where
\[
A^0 {\bf v}={\bf v} \mbox{ and } A^m {\bf v}=A\left( A^{m-1} {\bf v} \right) \mbox{ for } m \ge 1.
\]
Below we present matrix representation of the operator $A$ (which is an infinite matrix when $N=\infty$):
\[
A=\begin{bmatrix} p_0& p_1 & p_2 & \hdots & p_N \\0 & \frac{p_1}{q_1} & \frac{p_2}{q_1} & \hdots & \frac{p_N}{q_1} \\ 0 & 0 & \frac{p_2}{q_2} & \hdots & \frac{p_N}{q_2} \\ \vdots & \vdots & \vdots & \ddots & \vdots \\ 0 & 0 & 0 & \hdots & \frac{p_N}{q_N} \end{bmatrix} .
\]
Note that $A$ is an upper-triangular matrix with nonzero diagonal entries.
Let $e_m(l)=\mathbb{E} \left(W_{i+m} | W_i=l \right)$. Then, directly from the form of the conditional distribution it follows that
\begin{equation}
{\bf e}_{m+1}=A{\bf e}_m \hspace{6pt} \mbox{for} \hspace{6pt} m=1,2, \ldots
\end{equation}
In particular ${\bf e}_m$ is in the domain of $A$, given that ${\bf e}_{m+1}$ exists. Now we can state the theorem proposed in LB.
\begin{theorem}
\label{th:hipoteza}
Let $X$ be a random variable with discrete distribution with support $\lbrace 0,1,2, \dots , N \rbrace$ ($N \le \infty$). Let $(W_n)$ be the sequence of weak records built on a sequence $(X_n)$ of iid random variables having the same distribution as $X$. Assume that for some $i,s\ge 1$
\begin{equation}
\label{eq:zalozenie}
\mathbb{E} \left( W_{i+s} | W_i \right)=\beta_0 +\beta_1 W_i ,
\end{equation}
where $\beta_0, \beta_1 \in \mathbb{R}$. Then $\beta_0,\beta_1 >0$. Let $\gamma_0,\gamma_1$ be unique solutions of
\begin{equation}
\label{eq:rownania}
\beta_1=\gamma_1^s , \qquad \beta_0=\gamma_0 \frac{1-\gamma_1^s}{1-\gamma_1} .
\end{equation}
Then
\begin{enumerate}
\item if $0 < \beta_1 <1$, then $\frac{\gamma_0}{1-\gamma_1} \in \mathbb{N}$
\[
X \sim \mbox{\textup{nh}}_I \left( 1, \frac{\gamma_1}{1-\gamma_1},\frac{\gamma_0}{1-\gamma_1} \right) ,
\]
\item if $\beta_1=1$, then
\[
X \sim \mbox{\textup{geo}} \left(\frac{1}{1+\gamma_0} \right) ,
\]
\item if $\beta_1 >1 $, then
\[
X \sim \mbox{\textup{nh}}_{II} \left( 1, \frac{\gamma_0+1}{\gamma_1-1},\frac{\gamma_0}{\gamma_1-1} \right) .
\]
\end{enumerate}
\end{theorem}

The symbols of distributions above have the following meaning: $\mathrm{nh}_I$ is for the negative hypergeometric distribution of the first kind, $\mathrm{geo}$ is for the geometric distribution, $\mathrm{nh}_{II}$ is for the negative hypergeometric distribution of the second kind (more details on $\mathrm{nh}_I$ and $\mathrm{nh}_{II}$ laws can be found e.g. in Weso\l owski and Ahsanullah (2001)).

We will now recall basic steps in the proof given in LB. Observe,  that since $e_s(l)=\mathbb{E} \left(W_{i+s} | W_i=l \right)$ is strictly increasing, we have  $\beta_1 > 0$ and $\beta_0=e_s(0) >0$. Let $\gamma_0,\gamma_1$ be unique solutions of (\ref{eq:rownania}). Now, for $m=1, \dots, s$ we define ${\bf d}_m$ through the equality
\begin{equation}
\label{eq:zapis}
e_m(j)=\gamma_0 \frac{1-\gamma_1^m}{1-\gamma_1}+ \gamma_1^mj+d_m(j), \qquad j=0,1,\ldots, N.
\end{equation}
Directly from the definition of ${\bf d}_m$ and the assumption that ${\bf e}_s$ exists we obtain that ${\bf d}_m$ is in the domain of $A$ for $m=1,...,s-1$. From (\ref{eq:zalozenie}) we have that ${\bf d}_s=0$. After easy algebra we obtain
\begin{equation}
\label{wynik}
{\bf d}_{m+1}=\gamma_1^m{\bf d}_1+A{\bf d}_m, \qquad m=1,\ldots, s-1.
\end{equation}
From (\ref{wynik}) we can obtain that ${\bf d}_m$ is in the domain of $A^2$ for $m=1,...,s-1$ and by iterating (\ref{wynik}) we get that ${\bf d}_1$ is in the domain of $A^{s-1}$. This can be iterated and, consequently,
\begin{equation}
\label{eq:dowod}
{\bf d}_m=B_m{\bf d}_1, \qquad m=1,\ldots,s, \hspace{5pt}\quad \mbox{where}\quad B_m=\sum_{k=0}^{m-1}\gamma_1^{m-1-k}A^{k}.
\end{equation}
Let us note that $A$ and, consequently, $B_m$ depends on the unknown distribution ${\bf p}=(p_n)_{n\ge 0}$. To emphasize this fact, sometimes we will write $B_m^{({\bf p})}$ instead of $B_m$.

At this stage of argument we read in LB:
\begin{quotation}
\textit{Note that $B_m$ is an upper-triangular matrix with nonzero diagonal entries; then $B_m$ has an inverse (even in the infinite case)}
\end{quotation}
This is why, besides the case $N<\infty$ (equivalent to $\gamma_1\in(0,1)$), the proof is incorrect. The above statement is false in the case $N=\infty$ (that is $\gamma_1\ge 1$). Infinite matrices represent linear operators between linear spaces. In general such transformations, which are represented by infinite upper-triangular matrices with nonzero diagonal entries, do not have to be invertible and even injective. As an example consider a linear transformation $B:\R^{\infty}\to \R^{\infty}$ represented by the matrix
\[
B=\left[b_{ij} \right]_{i,j=0}^{\infty} \hspace{2pt} ; \hspace{5pt} b_{ij}=\begin{cases} 1 & \mbox{for } i=j \mbox{ or } j=i+1 \\ 0 & \mbox{otherwise}. \end{cases}
\]
Obviously, $B$ is an upper-triangular matrix with nonzero diagonal entries. Let $v=\left(1,-1,1,-1, \hdots \right)$. Then $Bv=0$ and thus $B$ is not injective in $\R^{\infty}$, consequently, it cannot be invertible. However, if we consider $B$ as a linear operator on the space of sequences convergent to $0$, then $B$ is invertible with $B^{-1}$ being also upper-triangular with $n$-th row of the form $(0,\ldots,0,1,-1,1,-1,\ldots)$, where the first $1$ is at the position $n$, $n\ge 1$.

In the next section we will discuss in detail injectivity of the operator $B_s$ defined in \eqref{eq:dowod}, which is of crucial importance since the rest of the argument from LB lies in plugging $m=s$ in (\ref{eq:dowod}). Since, as it was observed before, ${\bf d}_s=0$, it follows that
\[
B_s^{({\bf p})}{\bf d}_1=0.
\]
So if $B_s^{({\bf p})}$ was injective for any ${\bf p}\in\mathcal{M}_s$ we would get ${\bf d}_1=0$ and, consequently,
\[
e_1(j)=\gamma_0+\gamma_1j \quad \mbox{for} \quad j=0,1, \ldots , N .
\]

That is, the crucial problem for validity of the proof as suggested in LB is a question of injectivity of $B_s^{({\bf p})}$ for any ${\bf p}\in\mathcal{M}_s$.

\section{Is the operator $B_s^{({\bf p})}$ injective?}
In this section we will show how injectivity of  $B_s^{({\bf p})}$ defined on $D(B_s^{({\bf p})})=D(A^{s-1})$ depends on $s$. Let us recall that we are considering here only such distributions ${\bf p}$ for which $\mathbb{E}(W_{i+s}|W_i)<\infty$ and, as it has already been mentioned, this condition depends only on $s$ and not on $i$. First, we will consider operators with domains being subsets of $\mathbb{C}^{\infty}=\lbrace \left( x_0,x_1, \dots \right) : x_k \in \mathbb{C} \rbrace$,  the linear space of sequences of complex numbers.
\begin{theorem}
\label{th:iniekcja}
Let $B_s^{({\bf p})}$ be the operator defined by \eqref{eq:dowod} with $N=\infty$ on a domain $D(B_s^{({\bf p})})\subset \C^{\infty}$. Then $B_s^{({\bf p})}$ is injective for any distribution ${\bf p}\in\mathcal{M}_s$ iff $s \in \lbrace 2,3,4 \rbrace$.
\end{theorem}
\begin{proof}
Since $\sum_{k=0}^{s-1}z^k=\prod_{k=1}^{s-1}\left(z-\lambda_k \right)$, where $\lambda_k=\cos \left(\frac{2 k \pi}{s} \right) + i \sin \left(\frac{2 k \pi}{s} \right)$, $k=1,\ldots,s-1$, $s\ge 2$, we can represent the operator $B_s$ in the following way
\begin{equation}
\label{eq:prod}
B_s=\prod_{k=1}^{s-1}\left(A-\gamma_1\lambda_kI \right),
\end{equation}
where $I$ is the identity operator.
Thus if  $\gamma_1\lambda_{\ell}$ is an eigenvalue of $A$ for some $\ell =1,2,...,s-1$, then $B_s$ is not injective. Indeed, if ${\bf x}_{\ell} \in D \left(B_s\right)$ is a respective nonzero eigenvector of $\gamma_1\lambda_{\ell}$, then (note that $\left( A - \gamma_1 \lambda_j I\right)$ and $\left( A - \gamma_1 \lambda_k I\right)$ commute)
\[
B_s{\bf x}_{\ell}=\left[\prod_{\substack{k=1\\ k\ne \ell}}^{s-1}\left(A-\gamma_1\lambda_kI \right)\right](A-\gamma_1\lambda_{\ell}I){\bf x}_{\ell}=0.
\]
Consequently, $B_s$ is not injective.

Assume now that none of $\gamma_1\lambda_{\ell}$, $\ell=1,2,\dots , s-1$ is an eigenvalue of $A$. Then $B_s$ is a composition of injective operators, so $B_s$ must also be injective.

Finally, we conclude that $B_s$ is injective if and only if all $\gamma_1\lambda_{\ell}$, $\ell=1,2,\dots, s-1$, are not eigenvalues of $A$.

We will now examine eigenvalues of $A$ which are of the form $\lambda=\gamma_1\lambda_{\ell}$. Let $\lambda \in \mathbb{C}$, ${\bf x} \in D(A)$, ${\bf x} \neq 0$, be such that $A{\bf x}=\lambda {\bf x}$ which is equivalent to
\bel{pi}
\sum_{j=i}^{\infty}p_jx_j=\lambda x_iq_i \quad \forall\,i\ge 0 .
\ee
After subtracting the equality for $i$ and $i+1$ sidewise we obtain
\[
x_ip_i=\lambda \left(x_iq_i-x_{i+1}q_{i+1} \right).
\]
Hence we have
$$
x_{i+1}=\frac{\lambda q_i-p_i}{\lambda q_{i+1}}x_i \quad \forall\,i\ge 0.
$$
Expanding this recursion gives
\[
x_{i+1}=\frac{\lambda q_i-p_i}{\lambda q_{i+1}}\frac{\lambda q_{i-1}-p_{i-1}}{\lambda q_i}\ldots \frac{\lambda q_0-p_0}{\lambda q_1}x_0 \quad
\forall\,i\ge 0.
\]
We assumed that ${\bf x} \in D(A)$ and ${\bf x} \neq 0$ which now, given the expression above and \eqref{pi} for $i=0$, implies
\bel{eigcon}
\lambda \mbox{ is an eigenvalue of } A \Longleftrightarrow \sum_{k=1}^{\infty}\left| b_k\left( \lambda \right) \right| p_k < \infty \mbox{ and } 0=p_0-\lambda q_0 + \sum_{k=1}^{\infty} b_k\left( \lambda \right) p_k ,
\ee
where
\[
b_k\left( \lambda \right)=\prod_{i=0}^{k-1} \frac{\lambda q_i-p_i}{\lambda q_{i+1}}.
\]
Let us denote
\[
S_n\left( \lambda \right)=\begin{cases} p_0-\lambda q_0 & \mbox{for } n=0 \\ p_0-\lambda q_0 + \sum_{k=1}^nb_k\left( \lambda \right)p_k & \mbox{for } n \ge 1 \end{cases} \mbox{ and } S_n^{\ast}\left( \lambda \right)=\sum_{k=1}^n \left| b_k\left( \lambda \right) \right| p_k.
\]
By an easy induction argument we obtain the following product representation of $S_n\left( \lambda \right)$ for $n \ge 1$
\[
S_n\left( \lambda \right)=\left( p_0-\lambda q_0 \right)\prod_{k=1}^n\frac{\lambda q_k-p_k}{\lambda q_k}=\left( p_0-\lambda q_0 \right)\prod_{k=1}^n\left(1-\frac{c_k}{\lambda} \right), \mbox{ where } c_k=\frac{p_k}{q_k} \in \left( 0,1 \right).
\]
As observed in \eqref{eigcon}, we have to consider the situation when $\lim_{n \to \infty}S_n^{\ast}\left( \lambda \right) <\infty$ and $\lim_{n \to \infty} S_n\left( \lambda \right)=0$. Note that
\[
\lim_{n \to \infty}S_n\left( \lambda \right)=0 \Longleftrightarrow \lim_{n \to \infty}\left| S_n\left( \lambda \right) \right|^2=0 \Longleftrightarrow \lim_{n \to \infty} |p_0-\lambda q_0|^2\prod_{k=1}^n \left| 1-\frac{c_k}{\lambda} \right|^2=0.
\]
Since $\lambda=\gamma_1 \lambda_{\ell}$ for some $\ell=1,\hdots s-1$, $p_0-\lambda q_0 \neq 0$ (recall that $q_0=1$). Furthermore, we observe that $\frac{1}{\lambda_{\ell}}=\overline{\lambda_{\ell}}=\lambda_{s-\ell}$. Consequently
\[
\lim_{n \to \infty} S_n\left( \gamma_1 \lambda_{\ell} \right) =0 \Longleftrightarrow \prod_{k=1}^{\infty} \left| 1 - \frac{\lambda_{s-\ell}}{\gamma_1} c_k \right|^2=0.
\]
Since the single factor in the product has the form
\[
\left|1-\frac{\lambda_{s-\ell}}{\gamma_1}c_k \right|^2=1-2\frac{c_k}{\gamma_1}\cos\left(\frac{2 \left( s- \ell \right) \pi}{s} \right)+\left(\frac{c_k}{\gamma_1} \right)^2=1-2\frac{c_k}{\gamma_1}\cos\left(\frac{2  \ell \pi}{s} \right)+\left(\frac{c_k}{\gamma_1} \right)^2,
\]
we see that for all $k \ge 1$ it assumes the minimum for $\ell=1$. With that and \eqref{eigcon} in mind, we can conclude that $\gamma_1 \lambda_1$ is not an eigenvalue iff $\gamma_1 \lambda_{\ell}$ are not eigenvalues for any $\ell=1,2,\dots,s-1$ which leads to
\bel{equiv}
B_s^{({\bf p})} \mbox{ is not injective } \Longleftrightarrow \prod_{k=1}^{\infty} \left| 1 - \frac{\lambda_1}{\gamma_1} c_k \right|^2=0 \mbox{ and } \lim_{n \to \infty} S_n^{\ast}\left( \gamma_1 \lambda_1 \right) < \infty.
\ee
Thus we need to examine if the condition
\[
0=\prod_{k=1}^{\infty}\underbrace{\left( 1-2\frac{c_k}{\gamma_1}\cos\left(\frac{2 \pi}{s} \right)+\left(\frac{c_k}{\gamma_1} \right)^2 \right)}_{a_{k,s}}
\]
is satisfied.

Note that
\begin{itemize}
\item  $a_{k,2}=1+2\frac{c_k}{\gamma_1}+\left(\frac{c_k}{\gamma_1} \right)^2 > 1$,

\item  $a_{k,3}=1+\frac{c_k}{\gamma_1}+\left(\frac{c_k}{\gamma_1} \right)^2 > 1$,

\item $a_{k,4}=1+\left(\frac{c_k}{\gamma_1} \right)^2 > 1$.

\end{itemize}

Thus, in these three cases the above condition does not hold. Consequently, for any distribution ${\bf p}$ the operator $B_s^{({\bf p})}$ is injective for $s=2,3,4$.

Now consider $s \ge 5$ and a geometric distribution ${\bf p}$ with parameter $p \in (0,1).$ We choose the parameter $p$ in such a way that $ \cos \left( \frac{2 \pi}{5}\right)>\frac{p}{2}$. Note that for geometric distribution $c_k=p$, $k=1,2,\ldots$ Then
\[
2\cos \left( \frac{2 \pi}{s} \right) \ge 2\cos \left( \frac{2 \pi}{5} \right) >p \ge \frac{p}{\gamma_1},
\]
and thus
\[
2\frac{p}{\gamma_1} \cos \left( \frac{2 \pi}{s} \right)>\left( \frac{p}{\gamma_1} \right)^2,
\]
which yields
\[
1>\underbrace{1-2\frac{p}{\gamma_1} \cos \left( \frac{2 \pi}{s} \right)+\left( \frac{p}{\gamma_1} \right)^2}_{a_{k,s}}=const.
\]
Furthermore
\[
S_n^{\ast}\left( \gamma_1 \lambda_1 \right)=\sum_{k=1}^n \prod_{i=0}^{k-1} \left| \frac{\gamma_1 \lambda_1 q_i-p_i}{\gamma_1 \lambda_1 q_{i+1}} \right|p_k=\sum_{k=1}^n \frac{q_0p_k}{q_k}\prod_{i=0}^{k-1} \left| \frac{\gamma_1 \lambda_1 q_i-p_i}{\gamma_1 \lambda_1 q_i} \right|=p \sum_{k=1}^n \prod_{i=0}^{k-1} \sqrt{a_{k,s}}=p \sum_{k=1}^n \left( \sqrt{a_{1,s}} \right) ^k.
\]
Since $a_{k,s}=a_{1,s}<1$, we obtain that $\lim_{n \to \infty} S_n^{\ast}\left( \gamma_1 \lambda_1 \right) < \infty $. Therefore \eqref{equiv} yields that in the case of geometric distribution ${\bf p}$ with the parameter $p$ satisfying the inequality as above $B_s^{({\bf p})}$ for $s \ge 5$ is not injective.
\end{proof}

In Theorem \ref{th:iniekcja} we examined injectivity of $B_s^{({\bf p})}$ with domain $D(B_s^{({\bf p})})\subset\mathbb{C}^{\infty}$, but for the purposes of the problem we should only consider the injectivity or its lack on $D(B_s^{({\bf p})})\cap\R^{\infty}$. Of course, injectivity on $D(B_s^{({\bf p})})\subseteq \mathbb{C}^{\infty}$ implies injectivity on $D(B_s^{({\bf p})}) \cap \mathbb{R}^{\infty}$. Consequently, it follows from Theorem \ref{th:iniekcja}  that for any ${\bf p}$ the operator $B_s^{({\bf p})}$ is injective in $D(B_s^{({\bf p})}) \cap \mathbb{R}^{\infty}$ for $s=2,3,4$. The opposite implication may not be true in general but in the case we consider it turns out that it holds.

\begin{theorem}
\label{th:rzecz}
For any ${\bf p}$ the operator $B_s^{({\bf p})}$ is injective on $D(B_s^{({\bf p})}) \cap \mathbb{R}^{\infty}$ iff $s \in \lbrace 2,3,4 \rbrace$.
\end{theorem}
\begin{proof}
As already mentioned the implication $"\Leftarrow"$ is an immediate consequence of the same implication from Theorem \ref{th:iniekcja}.

We will prove the opposite implication by contradiction, i.e. we will show that there exists a distribution ${\bf p}$ such that for $s \ge 5$ the operator $B_s^{({\bf p})}$ is not injective in $D(B_s^{({\bf p})})\cap \mathbb{R}^{\infty}$. Let $s \ge 5$ and ${\bf p}=\left( p_k \right)_{k=0}^{\infty}$ be geometric distribution with parameter $p \in (0,1)$ such that $\cos \left( \frac{2 \pi}{5} \right) > \frac{p}{2}$. Then, from the proof of Theorem \ref{th:iniekcja}, we obtain that $\gamma_1 \lambda_1 , \gamma_1 \overline{\lambda_1}$ are eigenvalues of $A$. We denote a non-zero eigenvector attached to $\lambda=\gamma_1 \lambda_1$ by ${\bf x}=\left( x_0, x_1, \ldots \right)$ and the vector with conjugate entries by ${\bf \overline{x}}=\left( \overline{x}_0,\overline{x}_1, \ldots \right)$.

We will first note that ${\bf x}$ cannot be of the form ${\bf x}=i{\bf y}$ for a vector ${\bf y}\in\R^{\infty}$. Indeed, in such a case we would have $A{\bf y}=\lambda {\bf y}$ which is impossible since $\lambda$ is not a real number.

Note that ${\bf \overline{x}}$ is an eigenvector of $A$ attached to the eigenvalue $\overline{\lambda}$, because
\[
A{\bf \overline{x}}=\overline{A{\bf x}}=\overline{\lambda {\bf x}}=\overline{\lambda} {\bf \overline{x}},
\]
where the first equality holds since $A$ is a matrix with real entries. Consider
\[
B_s^{(\bf p)}\,{\bf x}=\prod_{k=1}^{s-1}\left( A - \gamma_1 \lambda_k I\right){\bf x}.
\]
Now, due to the fact that $\left( A-\gamma_1 \lambda_k I\right)$ and $\left( A-\gamma_1 \lambda_{\ell} I\right)$ commute, we obtain:
\[
B_s^{(\bf p)}\,{\bf x}=\left(\prod_{k=2}^{s-1}\left( A - \gamma_1 \lambda_k I\right)\right)\,\left( A - \gamma_1 \lambda_1 I\right){\bf x}=\prod_{k=2}^{s-1}\left( A - \gamma_1 \lambda_k I\right){\bf 0}={\bf 0}.
\]
The fact that, say, $\overline{\lambda_1}=\lambda_{s-1}$ and that ${\bf \overline{x}}$ is an eigenvector of $A$ respective to $\overline{\lambda}$ yield
\[
B_s^{(\bf p)}\,{\bf \overline{x}}=\prod_{k=1}^{s-2}\left( A - \gamma_1 \lambda_k I\right)\left( A - \gamma_1 \lambda_{s-1} I\right){\bf \overline{x}}=\prod_{k=1}^{s-2}\left( A - \gamma_1 \lambda_k I\right){\bf 0}={\bf 0}.
\]
Fix ${\bf z}={\bf x}+{\bf \overline{x}} \in \mathbb{R}^{\infty}$. Note that ${\bf z} \neq {\bf 0}$, because as it has already been observed, ${\bf x}$ cannot have all entries which are purely imaginery. Finally,
\[
B_s^{(\bf p)}\,{\bf z}=B_s^{(\bf p)}\,\left( {\bf x} + {\bf \overline{x}} \right)=B_s^{(\bf p)}\,{\bf x}+B_s^{(\bf p)}\,{\bf \overline{x}}={\bf 0}.
\]
Hence $B_s^{(\bf p)}$ is not injective in $D(B_s^{(\bf p)})\cap \mathbb{R}^{\infty}$ for $s \ge 5$.
\end{proof}
\section{Conclusion}
The above considerations on injectivity of $B_s^{(\bf p)}$ lead to the following correction to the result proposed in LB and recalled in Theorem \ref{th:hipoteza}.
\begin{proposition}
\label{c:prawd}
The assertion of Theorem \ref{th:hipoteza} holds true when $\gamma_1<1$ (that is, $N < \infty$) for any $s\ge 1$. For $\gamma_1 \ge 1$ (that is, $N= \infty$) the assertion of Theorem \ref{th:hipoteza} holds true for $s \in \lbrace 1,2,3,4 \rbrace$.
\end{proposition}
\begin{proof}
It is well known (see Section 1) that the result for $s=1$ holds true. The proof in the case $\gamma_1 <1$ (which implies $N < \infty$) given in LB is correct since in this case the operator  $B_s^{({\bf p})}$ is invertible for any $s\ge 2$. For $\gamma_1 \ge 1$ (which implies $N=\infty$) due to Theorem \ref{th:rzecz} we have injectivity of $B_s^{(\bf p)}$ for $s \in \lbrace 2,3,4 \rbrace$ therefore the method of the proof proposed in LB is correct and thus the respective part of the assertion from  Theorem \ref{th:hipoteza} holds true.
\end{proof}

Finally, let us emphasize that for $\beta_1\ge 1$ and $s\ge 5$ it follows from Theorem \ref{th:iniekcja} that $B_s^{(\bf p)}$ may not be injective for some distributions ${\bf p}\in\mathcal{M}_s$, even such that appear in the conclusion of Theorem \ref{th:hipoteza} (the geometric law was identified as such in the proof of Theorem \ref{th:iniekcja}) and thus the argument used in LB is no longer valid. {\bf Therefore the problem of characterization of the parent distribution of the sequence of iid observations from the discrete distribution by the condition $\E(W_{i+s}|W_i)=\beta_1W_i+\beta_0$ for $\beta_1\ge 1 $ and $s\ge 5$ remains open!}

\vspace{3mm}\noindent
{\bf Acknowledgement.} We are very thankful to referees for their remarks which helped us a lot in improving  presentation of the material of the paper. It will appear in {\em Statistics}.

\vspace{5mm}\noindent
{\bf References}
\begin{enumerate}

\item {\sc Ahsanullah, M., Hamedani, G.}, Characterizations of continuous distributions based on conditional expectations of generalized order statistics. {\em  Comm. Statist. Th. Meth.} {\bf 42(19)} (2013), 3608-3613.

\item {\sc Ahsanullah, M., Holland, B.}, Some properties of the distribution of record values from the geometric distribution. {\em Statist. H.} {\bf 25} (1984), 319-327.

\item {\sc Ahsanullah, M., Nevzorov, V.B.}, {\em Records via Probability Theory}, Atlantis Press, 2015.

\item {\sc Ahsanullah, M., Weso\l owski, J.}, Linearity of best predictors for non-adjacent record values. {\em Sankhya, Ser. B} {\bf 60} (1998), 221-227.

\item {\sc Aliev, F.A.}, Characterizations of distributions through weak records. {\em J. Appl. Statist. Sci.} {\bf 8} (1998), 13-16.

\item {\sc Aliev, F.A.}, Characterization of geometric distribution through weak records. In: {\em Asymptotic Methods in Probability and Statistics with Applications} (N. Balakrishnan, I.A. Ibragimov, V.B. Nevzorov, eds), Birkh\"auser, Boston 2001, 299-307.

\item {\sc Arnold, B.C., Balakrishnan, N., Nagaraja, N.H.} {\em Records}, Wiley, New York 1998.

\item {\sc Bairamov, I., Ahsanullah, M., Pakes, A.G.}, A characterization of continuous distributions via regression on pairs of record values. {\em Austral. N. Zeal. J. Statist.} {\bf 47(4)} (2005), 543-547.

\item {\sc Beg, M.I., Ahsanullah, M., Gupta, R.C.}, Characterizations via regressions of generalized order statistics. {\em Statist. Meth.} {\bf 12} (2013), 31-41.

\item {\sc Bieniek, M.}, On characterization of distributions by regression of adjacent generalized order statistics. {\em Metrika} {\bf 66(2)} (2007), 233-242.

\item {\sc Bieniek, M.}, A note on characterizations of distributions by regressions of non-adjacent generalized order statistics. {\em Austral. N. Zeal. J. Statist.} {\bf 51(1)} (2009), 89-99.

\item {\sc Bieniek, M., Szynal, D.}, Characterizations of distributions via
linearity of regression of generalized order statistics.  {\em Metrika} {\bf 58} (2003), 259-272.

\item {\sc Cramer, E., Kamps, U., Keseling, C.}, Characterizations via linear regression for ordered random variables - a unifying approach. {\em Comm. Statist. Th. Meth.} {\bf 33(12)} (2004), 2885-2911.

\item {\sc Danielak, K., Dembi\'nska, A.}, On characterizing discrete distributions via conditional expectations of weak records. {\em Metrika} {\bf 66(2)} (2007), 129-138.

\item {\sc Dembi\'nska, A., L\'opez-Bl\'azquez, F.}, A characterization of geometric distribution through $k$th weak records. {\em Comm. Statist. Th. Meth.} {\bf 34(12)} (2005), 2345-2351.

\item {\sc Dembi\'nska, A., Weso\l owski, J.}, Linearity of regression for non-adjacent order statistics. {\em Metrika} {\bf 48} (1998), 215-222.

\item {\sc Dembi\'nska, A., Weso\l owski, J.}, Linearity of regression for non-adjacent record values. {\em J. Statist. Plann. Infer.} {\bf 90} (2000), 195-205.

\item {\sc Ferguson, T.S.}, On characterizing distributions by properties of order statistics. {\em Sankhya Ser. A} {\bf 29} (1967),
265–278.

\item {\sc Franco, M., Ruiz, J.M.}, Characteriztion of discrete distributions based on conditional expectations of record values. {\em Statist. Pap.} {\bf 42} (2001), 101-110.

\item {\sc Gupta, R.C., Ahsanullah, M.}, Some characterization results based on the conditional expectation of a function of non-adjacent order statistics (record values). {\em Ann. Inst. Statist. Math.} {\bf 56(4)} (2004), 721-732.

\item {\sc Huang, W.-J., Su, J.-S.}, On certain problems involving order statistics - a unified approach through order statistics property of point processes. {\em Sankhya, Ser. A} {\bf 61} (1999), 36-49.

\item {\sc Kirmani, S.N.U.A., Alam, S.N.}, Characterization of the geometric distribution by the form of a predictor. {\em Comm. Statist. Theory Meth.} {\bf 9} (1980), 541-547.

\item {\sc Korwar, R.M.}, On characterizing distributions for which the second record value has a linear regression on the first. {\em Sankhya, Ser. B} {\bf 46} (1984), 108-109.

\item {\sc L\'opez-Bl\'azquez, F.}, Linear prediction of weak records. The discrete case. {\em Theory Probab. Appl.} {\bf 48(4)} (2004), 718-723.

\item {\sc L\'opez-Bl\'azquez, F., Moreno-Rebollo, J.L.}, A characterization of distributions based on linear regressions of order statistics and record values. {\em Sankhya, Ser. A} {\bf 59} (1997), 311-323.

\item {\sc L\'opez-Bl\'azquez, F., Weso\l owski, J.}, Discrete distributions for which the regression of the first record on the second is linear. {\em Test} {\bf 10} (2001), 121-131.

\item {\sc L\'opez-Bl\'azquez, F., Weso\l owski, J.}, Linearity of regression for the past weak and ordinary records. {\em Statistics} {\bf 38(6)} (2004), 457-464.

\item {\sc Nagaraja, H.N.}, On a characterization based on record values. {\em Austral. J. Statist.} {\bf 19} (1977), 70-73.

\item {\sc Nagaraja, H.N.}, Some characterizations of discrete distributions based on linear regressions of adjacent order statistics. {\em J. Statist. Plann. Infer.} {\bf 20} (1988), 65-75.

\item {\sc Nagaraja, H.N., Sen, P., Srivastava, R.C.}, Some characterizations of geometric tail distributions based on record values. {\em Statist. Pap.} {\bf 30} (1989), 147-155.

\item {\sc Nevzorov, V.B.}, {\em Records: Mathematical Theory}, AMS, Providence 2001.

\item {\sc Rao, R.C., Shanbhag, D.N.}, Recent result on characterizations of probability distributions: a unified approach through extensions of Deny's theorem. {\em Adv. Appl. Probab.} {\bf 18} (1986), 660-678.

\item {\sc Srivastava, R.C.}, Two characterizations of the geometric distribution by record values. {\em Sankhya, Ser. B} {\bf 40} (1979), 276-278.

\item {\sc Stepanov, A.V.}, A characterization theorem for weak records. {\em Theory Probab. Appl.} {\bf 38} (1994), 762-764.

\item {\sc Vervaat, W.}, Limit theorem for records from discrete distributions. {\em Stoch. Proc. Appl.} {\bf 1} (1973), 317-334.

\item {\sc Weso\l owski, J., Ahsanullah, M.}, Linearity of regression for non adjacent weak records. {\em Statist. Sinica} {\bf 11} (2001), 39-52.

\item {\sc Yanev, G.}, Characterization of exponential distribution via regression of one record value on two non-adjacent record values. {\em Metrika} {\bf 75} (2012), 743-760.

\end{enumerate}

\end{document}